\documentclass[11pt]{amsart}
\usepackage{amsmath,amssymb,amsthm}
\usepackage{mathrsfs, wrapfig,graphicx, caption, cite, pifont}
\usepackage{paralist}
\usepackage[all]{xy}
\usepackage{enumerate}
\usepackage{hyperref}
\usepackage{verbatim,color}
\numberwithin{equation}{section}

\newtheorem{theorem}{Theorem}[section]
\newtheorem{cor}[theorem]{Corollary}
\newtheorem{lemma}[theorem]{Lemma}

\newtheorem{prop}[theorem]{Proposition} \theoremstyle{definition}
\newtheorem{definition}[theorem]{Definition}
 \theoremstyle{remark}
\newtheorem{rem}[theorem]{Remark}

\newtheorem{exam}[theorem]{Example}
\newtheorem{lem}[theorem]{Lemma}

\def\beq{\begin{equation} } \def\eeq{\end{equation}}
\newcommand{\dd}[2]{\genfrac{}{}{0pt}{}{#1}{#2}}

\def\wh{\widehat}
\def\la{\lambda}
\def\wt{\widetilde}
 \def\cC{\mathcal C}

 \def\cB{ {\mathcal B} } \def\cC{ {\mathcal
C} }    \def\cG{ {\mathcal G}
}

\DeclareMathOperator{\sign}{sign}
\DeclareMathOperator{\sym}{Sym}
\DeclareMathOperator{\card}{card}
\DeclareMathOperator{\Ker}{ker}

\DeclareMathOperator{\Span}{span}

\def\RR{\mathbb R}

\def\eps{\varepsilon}  

\def\ben{\begin{enumerate} }

\def\een{\end{enumerate} }

  \makeatletter
  \newcommand{\mycontentsbox}{%
    {\centerline{NOT FOR PUBLICATION}
          \tableofcontents}}
  \def\enddoc@text{\ifx\@empty\@translators \else\@settranslators\fi
    \ifx\@empty\addresses \else\@setaddresses\fi
    \newpage\mycontentsbox}
  \makeatother

\begin{document}
\title{Sign patterns for chemical reaction networks}

\author[Helton]{J. William Helton${}^1$}
\address{J. William Helton, Mathematics Department,
University of California at San Diego, La Jolla CA 92093-0112}
\email{helton@ucsd.edu}
\thanks{${}^1$Partially supported by the NSF (DMS-0700758,
DMS-0757212) and the Ford Motor Co.}

\author[Katsnelson]{Vitaly Katsnelson}
\address{Vitaly Katsnelson, Mathematics Department,
University of California at San Diego, La Jolla CA 92093-0112}
\email{vkatsnel@ucsd.edu}

\author[Klep]{Igor Klep${}^2$}
\address{Igor Klep, Univerza v Ljubljani, Fakulteta za matematiko in fiziko,
Jadranska 19, SI-1111 Ljubljana\\
and Univerza v Mariboru, Koro\v ska 160, SI-2000 Maribor, Slovenia}
\email{igor.klep@fmf.uni-lj.si}
\thanks{${}^2$Financial support from the state
budget by the Slovenian Research Agency (project No. Z1-9570-0101-06).}

\date{\today}

\begin{abstract}
Most differential equations
found in chemical reaction networks (CRNs) have the form:
$$
\frac{dx}{dt}= f(x)= Sv(x),
$$
where $x \geq 0$, that is, $x$ lies in the nonnegative
orthant $\RR_{\geq 0}^d$, where
$S$ is a real $d\times d'$ matrix
(\emph{stoichiometric} matrix)  and $v$ is a column
vector consisting of $d'$ real-valued functions
having a special relationship to $S$.
Our main interest will be in the
Jacobian matrix, $f'(x)$, of $f(x)$,
in particular in whether or not each entry $f'(x)_{ij}$
has the same sign for all $x$ in the orthant,
i.e., the Jacobian respects a sign pattern.
 In other words species $x_j$ always acts on species $x_i$
in an inhibitory way or its action is always excitatory.

In \cite{HKG} we gave necessary and sufficient conditions
on the species-reaction graph naturally associated
to $S$ which guarantee that the
Jacobian of the associated CRN
has a sign pattern.
In this paper, given $S$ we give a construction
which adds certain rows and columns to $S$,
thereby producing a stoichiometric matrix
$\widehat S$ corresponding to a
new CRN with some added species and reactions.
The Jacobian for this CRN based on $\wh S$  has a sign pattern.
The equilibria for the $S$ and the $\wh S$ based CRN
are in exact one to one correspondence
with each equilibrium $e$ for the original CRN gotten
from an  equilibrium $\widehat e$ for the new CRN
by removing its added species.
In our construction of a new CRN we are allowed to choose rate
constants
for the added reactions
and if we choose them large enough
the equilibrium $\wh e$ is locally asymptotically stable if and only if the equilibrium
$e$ is locally asymptotically stable.
Further properties of the construction are shown,
such as those pertaining to conserved quantities
and to how the  deficiencies   of the two CRNs compare.
\end{abstract}

\maketitle

\section{Introduction}
\label{sec:intro}

In this paper we are concerned with polynomial systems of equations
arising from systems of
ordinary differential equations (ODEs) which act on the
nonnegative orthant $\RR_{\geq 0}^d$ in $\mathbb R^d$:
\beq
\label{eq:de}\frac{dx}{dt}=   f(x),
\eeq
where $f:\RR_{\geq 0}^d\to\RR^d$.
The differential equations we
address are of a special form found in chemical
reaction kinetics:
\beq\label{eq:chemde}
\frac{dx}{dt}=Sv(x),
\eeq
where $S$ is a real $d\times d'$ matrix and $v$ is a column
vector consisting of $d'$ real-valued functions.
An ODE \eqref{eq:de} has {\bf reaction form} provided
it is represented as in \eqref{eq:chemde} with $v(x)=\begin{bmatrix}
v_1 & \cdots & v_{d'}\end{bmatrix}^t$
and
\begin{equation}
\label{def:vdepends}
v_j \text{ depends exactly on variables } x_i \text{ for which }
S_{ij}<0.
\end{equation}
Since such ODEs are identified with chemical reaction networks, we often
refer to these as CRNs.
Call $S$ the {\bf stoichiometric matrix} and the entries of $v(x)$
the {\bf fluxes}.
We always assume the fluxes are continuously differentiable.
Furthermore, in many situations
all fluxes $v_j(x) $ are monotone nondecreasing
in each $x_i$ when the other variables are fixed, that is,
$v'(x)=\begin{bmatrix}
\frac{\partial v_i(x)}{\partial x_j}\end{bmatrix}_{ij}$,
the Jacobian of $v$, has all entries nonnegative
for all $x\in\RR^d_{> 0} 0$.
This happens in classical mass action kinetics
or for Michaelis-Menten-Hill type fluxes.
See \cite{BQ,Pa} for an exposition.
We shall develop a  few matrix theoretic phenomena
bearing on the properties of $f'(x)= S v'(x)$,
called the {\bf Jacobian of the reaction}.
For monotone nondecreasing fluxes
the reaction form  property \eqref{def:vdepends} is equivalent to
\beq
\label{eq:vdepjacI}
\frac{\partial v_j(x)}{\partial x_i}  \not\equiv 0\quad
\Leftrightarrow \quad \left(
\frac{\partial v_j(x)}{\partial x_i} \geq 0 \;\text{  and }
\not\equiv 0 \right)\quad
\Leftrightarrow \quad
S_{ij}<0.
\eeq
and this is what we shall mostly be using.

 The property analyzed in this paper is whether
 or not each entry $f'(x)_{ij}$ of the Jacobian of the CRN
 has an  unambiguous sign, that is it is the same for all
 $x \in \RR^{d}_{>0}$.
If the sign $f'(x)_{ij}$ is minus (resp.~plus), then the effect of species $j$
on species $i$ is always  inhibitory (resp.~excitatory).
 If such is the case we  say that
 $f'$ respects a sign pattern.
 
\subsection{Sign Pattern of $AA^t$}
\label{sec:aat}

Given this monotone property, we employ the language of signed matrices
\cite{BS}.
Call a {\bf sign pattern} a matrix $A$ with entries which are $\pm
a_{ij}$ or $0$, where $a_{ij}$ are free variables.
To a \emph{real} matrix $B$ we can associate its sign pattern $A={\rm SP}(B)$
with $\pm a_{ij}$ or $0$ in the correct locations.
Given a matrix with \emph{symbolic} entries (i.e., polynomials)
we might or might not be
able to associate a sign pattern. Here, we think of the free variables as
being \emph{positive}.

\begin{exam}
If $B=\begin{bmatrix}
0 & 6\\
-2 & -5
\end{bmatrix}$,
then $A={\rm SP}(B) = \begin{bmatrix}
0 & +a_{12}\\
-a_{21} & -a_{22}
\end{bmatrix}$
and
$$
AA^t=
\begin{bmatrix}
a_{12}^2 & -a_{12}a_{22}\\
-a_{12}a_{22} & a_{21}^2+a_{22}^2
\end{bmatrix}.
$$
Observe that $AA^t$ respects a
sign pattern.

On the other hand, if $B=\begin{bmatrix}
-1 & 6\\
-2 & -5
\end{bmatrix}$,
then $A={\rm SP}(B) = \begin{bmatrix}
-a_{11} & +a_{12}\\
-a_{21} & -a_{22}
\end{bmatrix}$
and
$$
AA^t=
\begin{bmatrix}
a_{11}^2+a_{12}^2 & a_{11}a_{21}-a_{12}a_{22}\\
a_{11}a_{21}-a_{12}a_{22} & a_{21}^2+a_{22}^2
\end{bmatrix}
$$
does \emph{not} respect a sign pattern. Namely, the off-diagonal entries
of $AA^t$ are not positive linear combinations of monomials in the $a_i$.
They may attain positive and negative values when evaluated at
appropriate positive values of the $a_{ij}$.
\end{exam}

\begin{theorem}[cf.~\protect{\cite[Theorem 5.1]{HKG}}]
\label{thm:super}
Let $A$ be a sign pattern. The hermitian square $AA^t$ of $A$
respects a sign pattern if and only if $A$ does not contain
a $2\times 2$ submatrix whose rows and columns can be permuted
to obtain a matrix whose sign pattern agrees with the one of
\begin{equation}\label{eq:2cycle3pm}
\begin{bmatrix}+1&-1\\-1&-1\end{bmatrix}
\quad\text{or}\quad
\begin{bmatrix}-1&+1\\+1&+1\end{bmatrix}.
\end{equation}
Such $ 2 \times 2$ matrices either  contain $3$ minus signs
 and $1$ plus sign,
or  they  contain $1$ minus sign and $3$ plus signs.
\end{theorem}

\begin{proof}[Proof of Theorem {\rm\ref{thm:super}}]
Suppose the entry $(AA^t)_{ij}$ of $AA^t$ fails to respect a sign. As
$$
(AA^t)_{ij}=\sum_k A_{ik} A^t_{kj} = \sum_k A_{ik} A_{jk},
$$
this is equivalent to not all terms of the last sum
having the same sign. Which is equivalent to
the existence of
$k,\ell$ with $\sign A_{ik}=\sign A_{jk}\neq 0$ and
$\sign A_{i\ell}=-\sign A_{j\ell}\neq 0$. Hence the $2\times 2$
submatrix
$$
\begin{bmatrix}
A_{ik} & A_{jk}\\
A_{i\ell} & A_{j\ell}
\end{bmatrix}
$$
of $A$ will (after a possible permutation of rows and columns)
have the same sign pattern as one of the matrices in \eqref{eq:2cycle3pm}.
\end{proof}

\subsection{Sign pattern for the Jacobian}
In the language of chemical reaction networks Theorem \ref{thm:super}
has an interpretation
as follows:

\begin{theorem}\label{thm:chemSP}
The Jacobian $f'(x)=Sv'(x)$
of the right hand side of
a reaction form ODE \eqref{eq:chemde} with monotone
nondecreasing fluxes
respects a sign pattern in the positive orthant whenever $S$
does not have a
$2\times 2$ submatrix whose rows and columns can be permuted
 to obtain
 a matrix whose sign pattern agrees with the one of
  \begin{equation}\label{eq:2cycle3m}
  \begin{bmatrix}+1&-1\\-1&-1\end{bmatrix}.
 \end{equation}

Conversely, under mass action kinetics $($cf.~{\rm \cite{BQ,Pa}} or
Section \S {\rm\ref{sec:equilib}}$)$, if 
a reversible stoichiometric matrix $S$ contains such a submatrix, then 
$f'(x)=Sv'(x)$ 
fails to respect a sign pattern
on an open dense set of
reversible matrices having the same sign pattern as $S$.
\end{theorem}

Since, it is brief we review  why this is true.
Write $S=S_+-S_-$ for real matrices $S_+$, $S_-$ with nonnegative
entries satisfying the complimentarity property
$(S_+)_{ij} (S_-)_{ij}=0$.
If the $(i,j)$th entry of $f'(x)=Sv'(x)$ does not have a
sign pattern, then $(S_+v'(x))_{ij}\ne 0$ and $(S_-v'(x))_{ij}\ne 0$.
As
$$(S_+v'(x))_{ij}=\sum_k (S_+)_{ik} v'(x)_{kj},$$
$(S_+v'(x))_{ij}\ne 0$ if and only if
for some $k$, $S_{ik} > 0$ and
$v'(x)_{kj}\not \equiv  0$,
so from the
reaction form property \eqref{eq:vdepjacI},
we get $S_{jk}<0$.
To summarize: $S_{ik} > 0$ and $S_{jk}<0$.
Similarly,
$(S_-v'(x))_{ij}\ne 0$ if and only if there is some $\ell$ with
$(S_-)_{i\ell} \ne 0$ and  $v'(x)_{\ell j} \not \equiv 0 $,
so we get $S_{j\ell}<0$.
To summarize: $S_{i\ell} < 0$ and $S_{j\ell}<0$.
Taken together this implies that the $2\times 2$ submatrix
of $S$ given by rows $i,j$ and columns $k,\ell$ has
the same sign pattern as
the matrix \eqref{eq:2cycle3m}, up to a permutation of rows and columns.
The converse reverses this  line of reasoning,
with the hypothesis requiring robustness under small perturbations
of $S$ to rule out fluke cancellations.

See \cite[\S 3.1]{HKG} for details and extensions.

That many CRNs have Jacobians with sign patterns is the basis 
for the works of Thomas \cite{Th}, 
Kaufman \cite{TK}, Soul\'e \cite{So}, Gouz\'e \cite{Go}, 
Sontag \cite{AS,ArS06,ArS08}, and many subsequent
publications. Typically Sontag and collaborators
assume this and something considerably
stronger to obtain results on {\it globally stable} equilibria.
Sontag has had the philosophy for many years that
any CRN can be modelled carefully
to have Jacobians respecting sign patterns.

This paper concerns CRNs whose Jacobians \emph{do not} have a sign pattern
and describes a method for transforming
such
a system of ODEs
into
 a system of ODEs whose Jacobian \emph{does} have a sign pattern,
and for which the equilibria of both CRNs remain ``the same''
 (see \S \ref{sec:equilib} for a precise formulation).
We shall refer to this as the sign fixing algorithm.

\section{Fixing the sign pattern for the Jacobian}
\label{sec:signpat}

In the first part of of this section,
\S \ref{sec:algor}, we describe our sign fixing algorithm.
In subsequent (sub)sections we show that the algorithm has
several (pleasant) properties.
Sections \S \ref{sec:eq} and \S \ref{sec:defic} 
show how equilibria and the classical notion
of deficiency behave with respect to our algorithm.

\subsection{An algorithm for eliminating non-signed entries of $f'$}
\label{sec:algor}

Let $S$ be a stoichiometric matrix associated to a chemical network
with a submatrix whose sign pattern coincides with 
that
of \eqref{eq:2cycle3m},
which Theorem \ref{thm:chemSP} demonstrates is an obstruction
for having a sign pattern.
Let $A,B$ be the species representing the two rows of $S$
corresponding to this {\bf bad submatrix}.
Consider the two columns of $S$ belonging to this bad
submatrix.

\bigskip

\hspace{2in}
\setlength{\unitlength}{1cm}
\begin{picture}(3,3)(-1.4,0)
\put(0,0){\line(1,0){3}}
\put(0,0){\line(0,1){3}}
\put(3,0){\line(0,1){3}}
\put(0,3){\line(1,0){3}}
\put(-0.4,2){$A$}
\put(-0.4,1){$B$}
\put(.6,2){$-p_1$}
\put(.9,1){$p_2$}
\put(1.6,2){$-p_3$}
\put(1.6,1){$-p_4$}
\put(.6,3.2){\eqref{eq:bad1}}
\put(1.7,3.2){\eqref{eq:bad2}}
\put(-1.5,1.5){$S=$}
\end{picture}

\noindent These yield two reactions in the network of the following form:
\begin{align}
                 \label{eq:bad1} p_1A+C_1 &\to p_2B + C_2 & \\
\label{eq:bad2}               p_3A+p_4B+C_3 & \to C_4,
                    \end{align}
where $p_i \in \mathbb{N}$
and $C_i$ are some (possibly empty) positive linear combinations
of species (avoiding $A$ and $B$).
\bigskip

We will construct
a new network from $S$ to eliminate this bad submatrix.
Consider the following network,
where each reaction of the original CRN remains
the same except that we add a new species $B'$, reaction \eqref{eq:bad1}
is replaced by
\begin{equation}\label{eq:bad1'}
 p_1A+C_1 \rightarrow B' + C_2,
 \end{equation}
and we create an additional reaction
\begin{equation}\label{eq:bad2'}
B' \rightarrow p_2B.
\end{equation}
Notice that the stoichiometric matrix $\check S$
associated to this new chemical network
will not have the bad submatrix we started with.
Also, no new bad submatrices have been added in this process. Thus we
have reduced the number of bad submatrices.

\bigskip

\bigskip

\hspace{2in}
\setlength{\unitlength}{1cm}
\begin{picture}(3,3)(5.6,0)
\put(5.4,1.5){$\check S=$}
\put(7,0){\line(1,0){3.7}}
\put(7,-0.7){\line(0,1){3.7}}
\put(10,-0.7){\line(0,1){3.7}}
\put(6.6,2){$A$}
\put(6.6,1){$B$}
\put(6.5,-0.47){$B'$}
\put(7.6,2){$-p_1$}
\put(8.6,1){$-p_4$}
\put(8.6,2){$-p_3$}
\put(10.15,1){$p_2$}
\put(7.6,3.2){\eqref{eq:bad1'}}
\put(8.7,3.2){\eqref{eq:bad2}}
\put(9.94,3.2){\eqref{eq:bad2'}}
\put(7,3){\line(1,0){3.7}}
\put(7.9,-.47){$1$}
\put(7.95,.98){$0$}
\put(10.1,-.47){$-1$}
\put(7,-0.7){\line(1,0){3.7}}
\put(10.7,-0.7){\line(0,1){3.7}}
\put(7.1,-.47){$0$}
\put(8.9,-.47){$0$}
\put(9.7,-.47){$0$}
\put(10.25,0.15){$0$}
\put(10.25,1.9){$0$}
\put(10.25,2.6){$0$}
\end{picture}
\vspace{.9cm}

\noindent 
We continue applying the same procedure (on this new network)
to eliminate any other existing bad
$2\times 2$ submatrices.

In matrix terms, each time we apply this procedure to eliminate a bad submatrix, we change one column (e.g.~changing reaction \eqref{eq:bad1} into
\eqref{eq:bad1'}) of $S$ and we append one additional row (e.g.~for the
``species'' $B'$)
and column (e.g.~for the reaction \eqref{eq:bad2'}) to $S$. We will call this procedure the \textbf{sign fixing algorithm}.

\begin{definition} Let $S$ be a stoichiometric matrix corresponding to a chemical network, and suppose $S$ has bad submatrices.
We write $\widehat{S}$ for a new stoichiometric matrix with no bad submatrices obtained by the sign fixing algorithm applied to each bad submatrix as explained above.
$\widehat{S}$ is called a \textbf{sign fixing matrix} of $S$.
\end{definition}

Each step of the sign fixing algorithm has the interpretation
that we are keeping track of additional information.
Namely, we measure how much of species $B$ is produced
by reaction \eqref{eq:bad1} thereby obtaining $B'$
and also we measure how much of species $B$
is consumed by  reaction \eqref{eq:bad2}.
Of course,
in the original CRN we just kept track of the net
amount of $B$, a single species, now we have replaced this with two
species.
The added reaction has the effect of identifying
the two species asymptotically.

The above definition implies that if $S$ has no bad submatrices,
then $S=\widehat S$ is the sign fixing matrix of itself.
If $S$ has multiple bad submatrices, we obtain $\widehat S$
by a finite number of applications of the sign fixing algorithm.

We emphasize what we have found so far by stating

\begin{theorem}
Given a CRN one can derive a sign fixed CRN.
The Jacobian of the sign fixed CRN respects a sign pattern,
provided each flux $v_j$ is   monotone increasing on each $x_i$.
\end{theorem}

This follows from Theorem \ref{thm:chemSP} because 
we have produced a sign fixed CRN $\wh S$ 
not containing the forbidden pattern \eqref{eq:2cycle3m}.

\begin{exam}
\label{ex:sf}
We now present a modification of Example
\cite[Table 1.1.(v)]{CF05}
 of a CRN who's Jacobian fails to have a sign pattern,
 and we illustrate  use of  the sign fixing algorithm
 to create a network who's Jacobian has a sign pattern.
Consider the network:
\begin{eqnarray}
\label{eq:reaction1}
A+B &\to & F 
\\
\label{eq:reaction2}
A+C &\to & G
\\
\label{eq:reaction3}
C+D &\rightleftharpoons&  B
\\
\label{eq:reaction4}
C+E & \rightleftharpoons & 2D
\end{eqnarray}
\def\bo{ {\large \bf{1} } }
\def\bt{{\large \bf{2} } }
The stoichiometric matrix for this network (with species
arranged in alphabetical order) is
$$S=\begin{bmatrix}
    -1 &  -1 &  0 & 0 & 0 & 0\\
    -1 &  0 &  1 & -1 & 0 & 0\\
     0 &  -1 &  - \bo  & 1 & - \bo &  \bo \\
     0 &  0 &  - \bo  & 1 &   \bt & - \bt  \\
     0 &  0 &  0 & 0 & -1 & 1\\
     1 &  0 &  0 & 0 & 0 & 0\\
    0 &   1 & 0 & 0 & 0 & 0
    \end{bmatrix}.
    $$
The boldface entries are those appearing in bad submatrices of $S$,
of which there are two:
\begin{itemize}
\item 
$B_1$, corresponds
to species $C,D$ and the forward reactions of
\eqref{eq:reaction3} and \eqref{eq:reaction4}.

\item
$B_2$, corresponds to species $C,D$ and the
forward reaction of      \eqref{eq:reaction3}
and the reverse  reaction of \eqref{eq:reaction4}.
\end{itemize}

\noindent
The first step of the sign fixing algorithm eliminates $B_2$ and gives
$$ \check{S}
= \begin{bmatrix}
    -1 & -1 &  0   & 0 & 0   & 0 & 0\\
    -1 &  0 &  1   &-1 & 0   & 0 & 0\\
     0 & -1 & -\bo & 1 &-\bo & \bo & 0\\
     0 &  0 & -\bo & 1 & 0   &-\bt & 2 \\
     0 &  0 &  0 & 0 &-1 & 1 & 0\\
     1 &  0 &  0 & 0 & 0 & 0 & 0\\
     0 &  1 &  0 & 0 & 0 & 0 & 0\\
     0 &  0 &  0 & 0 & 1 & 0 & -1
    \end{bmatrix}.
$$
The second step eliminates $B_1$ and gives
$$ \wh{S}= \check{\check S}
= \begin{bmatrix}
    -1 & -1 &  0 & 0 & 0 & 0 & 0 & 0\\
    -1 &  0 &  1 &-1 & 0 & 0 & 0 & 0\\
     0 & -1 & -\bo & 1 &-\bo & 0   & 0 & 1\\
     0 &  0 & -\bo & 1 & 0   &-\bt & 2 & 0\\
     0 &  0 &  0 & 0 &-1 & 1 & 0 & 0\\
     1 &  0 &  0 & 0 & 0 & 0 & 0 & 0\\
     0 &  1 &  0 & 0 & 0 & 0 & 0 & 0\\
     0 &  0 &  0 & 0 & 1 & 0 &-1 & 0\\
     0 &  0 &  0 & 0 & 0 & 1 & 0 &-1
    \end{bmatrix},
    $$
Thus $\wh{S}$ is a sign fixing matrix for $S$.
\qed
\end{exam}

\subsection{The story in terms of graphs}

To $S$ one often associates a bipartite graph $\cG_S$.
One set of nodes is rows
R$i$ (chemical species)
the other set of nodes is
columns C$j$ (reactions).
If $S_{ij}$ has a $-$ sign (resp.~$+$),
then the arrow points from R$i$ into C$j$ (resp.~out of C$j$ into R$i$).
The arrow is from R$i$ into C$j$ (resp. out of C$j$ into R$i$) 
if species $i$ is consumed
(resp.~produced) in reaction $j$, respectively.
Which leads us
to refer to {\bf consumed} and {\bf produced edges}.
Also we make a distinction between dotted and solid edges,
the convention being that all dotted edges touching a particular C$j$
are either all consumed or all produced.
Likewise for solid edges. This does not determine the choice
solid vs.~dotted uniquely, for example,
completely switching the choice
of solid and dotted carries the same information.
If no reaction is reversible this would be redundant information with
the direction of arrows, but for reversible reactions dotted vs.~solid
is needed.

This graph is a simplified version of the species-reaction graph used
in \cite{CF06}.
Theorem \ref{thm:super} (see also the paragraph
following it) in this languages says

\begin{theorem}
The Jacobian $f'(x)=Sv'(x)$ of the right hand side of
a reaction form ODE \eqref{eq:chemde} with monotone
nondecreasing fluxes respects a sign pattern in the positive
orthant whenever the graph $\cG_S$
does not
contain a cycle of length four with three
consumed edges
 and one
 produced edge. That is the CRN contains two reactions and two species,
 one reaction consumes both species
 while the other reaction consumes one species and produces the other.
 \end{theorem}

 The sign fixing algorithm takes a length four cycle $\cC$ as
 in the theorem and ``breaks it'' by
 \ben[\rm (1)]
 \item
 removing an edge from the cycle
 \item
 adding a reaction node C$^*$ and species node R$^*$
 and two edges to the graph
 \een
 thereby converting $\cC$ to ``harmless'' a cycle of length six.

 \begin{exam} We return to
 Example \ref{ex:sf} and observe that its graph is
   
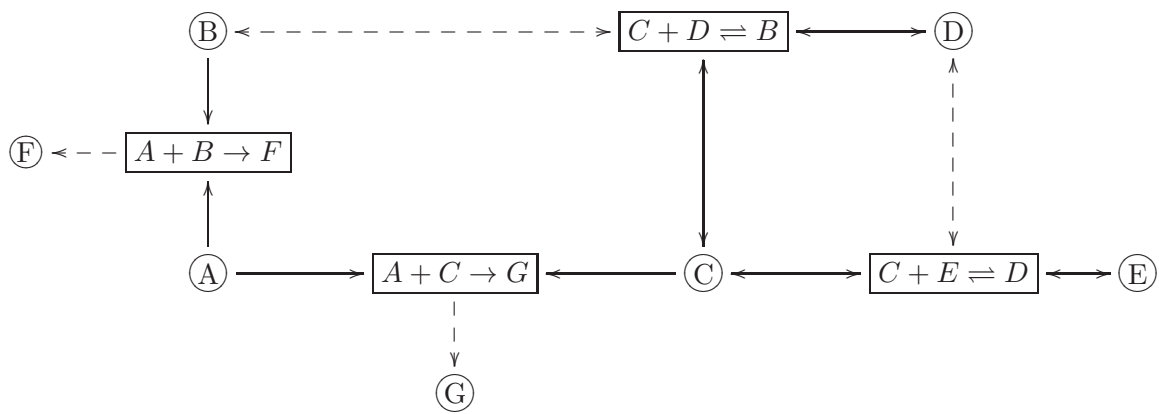
\[ \xymatrix{
   & \begin{xy} *+{ \txt{B} }* \frm{o} \end{xy} \ar@{{}-{>}}[d]
   \ar@{{<}--{>}}[rr] & &
   \fbox{$C+D \rightleftharpoons  B$} \ar@{{<}-{>}}[dd] \ar@{{<}-{>}}[r] &
   \begin{xy} *+{ \txt{D} }* \frm{o} \end{xy} \ar@{{<}--{>}}[dd] \\
   \begin{xy} *+{ \txt{F} }* \frm{o} \end{xy} \ar@{{<}--{}}[r] &
   \fbox{$A+B \to  F$} \ar@{{<}-{}}[d] \\ & \begin{xy} *+{ \txt{A} }*
   \frm{o} \end{xy} \ar@{{}-{>}}[r] &
   \fbox{$A+C \to  G$} \ar@{{<}-{}}[r]
   \ar@{{}--{>}}[d] & \begin{xy} *+{ \txt{C} }* \frm{o} \end{xy}
   \ar@{{<}-{>}}[r] &
   \fbox{$C+E  \rightleftharpoons  D$} \ar@{{<}-{>}}[r] &
   \begin{xy} *+{
   \txt{E} }* \frm{o} \end{xy} \\ & & \begin{xy} *+{ \txt{G} }*
   \frm{o} \end{xy} } \]
   (Instead of R$i$ and C$j$ more descriptive names have been used to denote
   the nodes.)

   \end{exam}

\subsection{Uniqueness of the sign fixing matrix}\label{subsec:nonuniq}

Suppose we have a CRN with stoichiometric matrix $S$.
The sign fixing matrix $\widehat S$ of $S$ is non-unique, since
it depends on the indexing of the bad submatrices of $S$.
It turns out that this non-uniqueness is easy to classify.
One finds that any two sign fixing matrices for $S$ can be
gotten from the other by  ``conjugation'' with a permutation matrix
in a certain class which we shall describe completely.

Suppose $B=\{b_1,\dots,b_n\}$ are the bad submatrices of $S$.
Then $\widehat{S}$ is determined by the finite sequence
$S=S_0,S_1,\ldots,S_{n-1},S_n=\widehat{S}$, where $S_j$ is the sign fixing matrix of $S_{j-1}$ with respect to $b_j$.
Another problem arises if $b_i$ and $b_j$ share the same positive entry in $S$. Namely, in this case, $S_j$ will equal $S_{j-1}$ by construction, if $i<j$.
To resolve this problem, we introduce an equivalence relation on $B$:
$b_i \thicksim b_j$ if and only if $b_i$ and $b_j$ share a positive entry of $S$.
This is an equivalence relation by construction. We will use $\mathcal{B}$ to denote
the set of all equivalence classes of bad submatrices of $S$, and from now on, we identify each bad submatrix of $S$ with its equivalence class.
Thus, the number of new columns and rows in $\widehat{S}$ is precisely $\card(\mathcal{B})$.

Suppose $\mathcal{B}=\{b_1,\dots,b_n\}$. Let $\sym_n$ be the symmetric group
on $\{1,2,\ldots,n\}$, i.e., the set of all permutations of $n$ elements.
Every $\sigma \in \sym_n$ determines a sign fixing matrix
$\widehat S_{\sigma}$ of $S$ by the finite sequence $S,S_{\sigma,1},\ldots,S_{\sigma,n-1},S_{\sigma,n}=\widehat S_{\sigma}$, where $S_{\sigma,1}$ is the sign fixing matrix of $S$ with respect to $b_{\sigma(1)}$,
and $S_{\sigma,j}$ is the sign fixing matrix of $S_{\sigma,j-1}$ with respect to $b_{\sigma(j)}$.
Clearly, each sign fixing matrix of $S$ is determined by a permutation of $\cB$,
and hence by an element in $\sym_n$. We thus identify the set of all sign fixing matrices of $S$ with the elements of $\sym_n$.

Our result in this subsection gives
a map from one sign fixing matrix to another.

    \begin{theorem}\label{thm:nonuniq}
    Suppose $S$ is a $d \times d'$ stoichiometric matrix corresponding to a CRN
    with $n$ pairwise nonequivalent bad submatrices. If $\sigma, \tau \in \sym_n$ then
    $$ \widehat S_{\sigma}= \begin{bmatrix}
                I_d & 0 \\ 0 & P
                \end{bmatrix}
                \widehat S_\tau \begin{bmatrix}
                                I_{d'} & 0 \\ 0 & P
                                                \end{bmatrix}^t
                                                        ,$$
    where $P$ is the $n \times n$ permutation matrix associated to $\tau^{-1}\sigma\in\sym_n$.
    \end{theorem}

Let $C$ be the matrix obtained by substituting
each of the $n$ positive entries of $S$ corresponding to the $n$ bad submatrices by $0$. By construction, each sign fixing matrix of $S$ will have the form
$$S_{\sigma}=\begin{bmatrix}
    C & M_{\sigma} \\
    N_{\sigma} &  -I_n
    \end{bmatrix}, $$
    where $N_{\sigma}$ is a $n \times d'$ matrix whose rows correspond to the bad submatrices in the order determined by $\sigma$, and $M_{\sigma}$ is a $d \times n$ matrix whose columns correspond to the bad submatrices in the order determined by $\sigma$.
More precisely, the $i$th row of $N_{\sigma}$ is the unit vector
$e_j'$ of length $d'$ if the bad submatrix $b_{\sigma(i)}$ has a positive
entry in column $j$ of $S$.
Similarly, the $i$th column of $M_{\sigma}$ is a multiple of the unit vector
$e_j$ of length $d$ if the bad submatrix $b_{\sigma(i)}$ has a positive
entry in row $j$ of $S$. The multiple is the value of $S$ at this positive
entry.

\begin{proof}[Proof of Theorem {\rm\ref{thm:nonuniq}}]
Note that
$$
\begin{bmatrix}
                I & 0 \\ 0 & P
                                \end{bmatrix}
                                                \widehat S_\tau \begin{bmatrix}
                                                                                I & 0 \\ 0 & P
                                                                                                                                \end{bmatrix}^t=
                                                                                                                            \begin{bmatrix}
                                                    C & M_{\tau}P^t \\
                                                        P N_{\tau} &  -I_n
                                                            \end{bmatrix},$$
so we only need to prove $M_{\tau}P^t=M_{\sigma}$ and $PN_{\tau}=N_{\sigma}$
for the given $P$.

Suppose $\mathcal{B}=\{b_1,\dots,b_n\}$, and $\eps \in \sym_n$ is the identity permutation. With the notation above, let
$$N_\eps =\begin{bmatrix}
    \alpha_1  \\
    \vdots \\
     \alpha_n
    \end{bmatrix}\quad \text{and} \quad
    M_\eps=\begin{bmatrix}
    \beta_1 & \cdots & \beta_n
    \end{bmatrix},$$
    where $\alpha_i$, $i=1,\ldots,n$ are the rows of $N_\eps$, and $\beta_i$, $i=1,\ldots,n$ are the columns of $M_\eps$. Given $\sigma \in \sym_n$, we have
    $$N_{\sigma} =\begin{bmatrix}
    \alpha_{\sigma(1)}  \\
    \vdots \\
     \alpha_{\sigma(n)}
      \end{bmatrix}\quad \text{and} \quad
    M_{\sigma}=\begin{bmatrix}
    \beta_{\sigma(1)} & \cdots & \beta_{\sigma(n)}
    \end{bmatrix}.$$

The $n \times n$ permutation matrix associated to $\tau^{-1}\sigma$ is
$$P=\begin{bmatrix}
    e_{(\tau^{-1}\sigma)(1)}  \\
    \vdots \\
     e_{(\tau^{-1}\sigma)(n)}
\end{bmatrix},$$
where $e_i$ denotes the unit vector of length $n$ with a
one in the $i$th coordinate and $0$'s elsewhere.
    By construction, $PN_{\tau}$ is a matrix whose $i$th row is the
    $ (\tau^{-1}\sigma)(i)$th row of $N_{\tau}$,
    so the $i$th row of $PN_{\tau}$ is $\alpha_{\tau((\tau^{-1}\sigma)(i))}= \alpha_{\sigma(i)}$. Hence $N_{\sigma}= PN_{\tau}$ as desired.

To conclude the proof let us verify $M_{\sigma} = M_{\tau}P^t$. Notice that $M_{\tau}P^t$ is a matrix whose $i$th column is the $(\tau^{-1}\sigma)(i)$th column of $M_{\tau}$, so the $i$th column of $M_{\tau}P^t$ is $\beta_{\tau((\tau^{-1}\sigma)(i))}= \beta_{\sigma(i)}$. This implies $M_{\sigma} = M_{\tau}P^t$.
\end{proof}

\section{Equilibria behave well under sign fixing}\label{sec:eq}

Our next goal is to analyze how equilibria for the original
CRN
compare to equilibria for a sign fixed CRN.
We shall find
that the equilibria are in perfect correspondence.
The key to this is a simple fact in linear algebra
which constitutes the next subsection.

\subsection{Linear algebra associated to sign fixing}

We now show how
 the nullspace of $S$ and the nullspace of $\check S$
 are related.
 Likewise for
 the range of $S$ vs.~the range of $\check S$.

\begin{prop}\label{prop:ker}
Let $S$ be a stoichiometric matrix with a bad submatrix.
Let $\check S$ be obtained from $S$ by applying the
sign fixing algorithm to eliminate this bad submatrix.
Then $\dim\ker S=\dim\ker\check S$ and $\dim\ker S^t=\dim\ker\check S^t$.

Indeed, there is a precise correspondence:
 given $v\in\ker S$ there exists a unique $v_\infty\in\RR$ with
$\check v=\begin{bmatrix} v^t & v_\infty\end{bmatrix}^t
\in\ker\check S$. Conversely, for  $\check v
=\begin{bmatrix} v^t & v_\infty\end{bmatrix}^t
\in\ker\check S$, one has $v\in \ker S$.
A similar statement holds for the left kernels.
Furthermore, under this correspondence $v$ has positive
$($resp.~nonnegative$)$ entries if and only if $\check v$ does. 
\end{prop}

\begin{proof}
Let $p$ and $q$ be the rows of $S$,
and let $k$ and $\ell$ be the columns of
$S$ corresponding to the bad submatrix.
Without loss of generality assume that
the bad submatrix has the same sign pattern as
$\left[\begin{smallmatrix}
     -1 & -1\\
          -1 & +1
          \end{smallmatrix}\right]$.
For the sake of exposition, here is a picture:\\
\bigskip

\setlength{\unitlength}{1cm}
\begin{picture}(3,3)(-1.4,0)
\put(0,0){\line(1,0){3}}
\put(0,0){\line(0,1){3}}
\put(3,0){\line(0,1){3}}
\put(0,3){\line(1,0){3}}
\put(-0.3,2){$p$}
\put(-0.3,1){$q$}
\put(1,3.1){$k$}
\put(2,3.1){$\ell$}
\put(1.8,1){$S_{q,\ell}$}
\put(-1.5,1.5){$S=$}
\put(4.1,1.5){$\Rightarrow$}
\put(5.4,1.5){$\check S=$}
\put(7,0){\line(1,0){3.7}}
\put(7,-0.7){\line(0,1){3.7}}
\put(10,-0.7){\line(0,1){3.7}}
\put(7,3){\line(1,0){3.7}}
\put(6.7,2){$p$}
\put(6.7,1){$q$}
\put(6.05, -0.5){$d+1$}
\put(8,3.1){$k$}
\put(9,3.1){$\ell$}
\put(10,3.1){\tiny{$d'+1$}}
\put(8.9,1){$0$}
\put(8.9,-.47){$1$}
\put(10.1,-.47){$-1$}
\put(7,-0.7){\line(1,0){3.7}}
\put(10.7,-0.7){\line(0,1){3.7}}
\put(10.05,1){$S_{q,\ell}$}
\put(7.1,-.47){$0$}
\put(8,-.47){$0$}
\put(9.7,-.47){$0$}
\put(10.25,0.15){$0$}
\put(10.25,1.9){$0$}
\put(10.25,2.6){$0$}
\end{picture}
\vspace{.9cm}

Suppose $S\in\RR^{d\times d'}$ and let
$v\in\ker S$. Then clearly
$\check v=\begin{bmatrix} v^t & v_\ell\end{bmatrix}^t
\in\ker\check S$. Moreover, if $v\in\RR_{>0}^{d'}$, then
$\check v\in\RR_{>0}^{d'+1}$.
Conversely, every
$\check v
=\begin{bmatrix} v^t & v_\infty\end{bmatrix}^t \in\ker\check S$
satisfies $\check v_\ell=v_\infty$ and so gives rise
to $v\in\ker S$.
Again, positivity is preserved.
The corresponding analogous statements and proofs for
the left kernel are left as
an exercise for the reader.
\end{proof}

\begin{rem}
\label{rem:sflessbad}
After applying one step of the sign fixing algorithm
the number of bad submatrices decreases.
More precisely, let $S$ be a stoichiometric matrix with a bad
submatrix. Let $\check S$ be obtained from $S$ by applying the
sign fixing algorithm to eliminate this bad submatrix.
Then the number of bad submatrices in $\check S$ is less than the
number of those in $S$.

Indeed, the form of $\check S$,
because the added row
and column have all entries but two equal to zero,
guarantees $\check S$ does not contain  the one bad submatrix under attack,
and at the same time no new bad submatrices have been added.
Thus we have  reduced the number of bad submatrices by at least one.\qed
\end{rem}

\begin{rem}
The
sign fixing algorithm for the situation of the hermitian square
$AA^t$
of a sign pattern $A$,
goes just as in \S \ref{sec:algor}
with $A$ replacing $S$ in the picture in
the proof of Proposition \ref{prop:ker}.

Details are
left as an exercise for the interested reader. \qed
\end{rem}

\subsection{Behavior of equilibria and steady states in mass action kinetics
under sign fixing}
\label{sec:equilib}

This subsection uses the linear algebra result of the previous
subsection to show
that the equilibria of a CRN and of its sign fixed CRN are in perfect
correspondence.
We shall show this for mass action kinetics,
although as one will see from the
arguments here it works for a much more general class of CRNs.

We now review mass action kinetics with the primary aim of introducing
our notation.
The postulate of mass action kinetics is
\emph{``the reaction rate is proportional to reactant concentrations''}.
For instance, for the chemical reaction
$$
2A+B\to 4C
$$
the reaction rate is $k_{\scriptscriptstyle 2A+B\to 4C} x_A^2 x_B,$
where $x$ denotes the {\bf concentration} of a species and
$k_{\scriptscriptstyle 2A+B\to 4C}>0$ is the {\bf rate constant}.
The corresponding ODE is
\begin{equation*}
\begin{bmatrix}
\dot x_A \\
\dot x_B \\
\dot x_C
\end{bmatrix} =
\begin{bmatrix}
 -2 k_{\scriptscriptstyle 2A+B\to 4C} x_A^2 x_B\\
  - k_{\scriptscriptstyle 2A+B\to 4C} x_A^2 x_B \\
   4 k_{\scriptscriptstyle 2A+B\to 4C} x_A^2 x_B.
   \end{bmatrix} =
   \begin{bmatrix}
   -2\\ -1\\4
   \end{bmatrix}
   \begin{bmatrix}
   k_{\scriptscriptstyle 2A+B\to 4C} x_A^2 x_B
   \end{bmatrix}.
   \end{equation*}
In general, for $S\in\RR^{d\times d'}$ the flux vector $v(x)$ is
given by
$$
v(x)_i=k_i \prod_{j=1}^{d} x_j^{- \min \{0,S_{ji}\}}   ,\quad i=1,\ldots,d'.
$$
(Here $k_i>0$ is the rate constant associated to the $i$th reaction
and $x_j$ is the concentration of the $j$th species.)

If the ODE \eqref{eq:de} admits a positive vector in the left kernel
(i.e., there exists
$m\in \RR_{>0}^{d}$ with $m\cdot\dot x= m\cdot f(x)=0$), then
the ODE is called {\bf conserving}.
This reflects quantities (like the mass or
the number of carbon atoms) being conserved.
An obvious sufficient condition for ODEs of the form \eqref{eq:chemde}
is $\ker S^t\cap \RR_{>0}^{d}\neq\{0\}$.
If this is satisfied, we say that $S$ is {\bf conserving}.
By Proposition \ref{prop:ker} this condition is preserved under the
sign fixing algorithm.

\begin{cor}\label{cor:eq}
Let $S$ be a stoichiometric matrix, and suppose $S$ has a bad submatrix.
Let $\check S$ be obtained from $S$ by applying the
sign fixing algorithm to eliminate this bad submatrix.
If $S$ is conserving, then so is $\check S$.
Moreover, the reaction form
differential equations \eqref{eq:chemde} associated
under mass action kinetics
to $S$ and to $\check{S}$,
respectively, have the
same
equilibria in the following sense.

Suppose $S\in\RR^{d\times d'}$. If
$\check x=\begin{bmatrix} x^t &x_\infty\end{bmatrix}^t\in\RR_{>0}^{d+1}$
$($resp.~$\check x\in\RR_{\geq0}^{d+1})$
satisfies $\check S \check v(\check x)=0$, then
$S v(x)=0$. Conversely, if $x\in\RR_{>0}^d$ $($resp.~$x\in\RR_{\geq0}^d)$ 
satisfies $S v(x)=0$, then
there exists a unique $x_\infty\in\RR_{>0}$ $($resp.~$x_\infty\in\RR_{\geq0})$ 
with
$\check S \check v(\check x)=0$ for
$\check x=\begin{bmatrix} x^t &x_\infty\end{bmatrix}^t$.
$($Here  $\check v$ will be used to denote a
flux vector associated
under mass action kinetics
to $\check S$.$)$
\end{cor}

\begin{proof}
This is essentially a consequence of Proposition \ref{prop:ker}
and the figure contained in its proof describes the notation we now
use.
Let $p$ and $q$ be the rows of $S$,
and let $k$ and $\ell$ be the columns
of $S$ corresponding to the bad submatrix,
and assume without loss of generality
$S_{q\ell}>0$.

If $\check v(\check x)\in \ker\check S\cap\RR_{>0}^{d'+1}$ and
$\check x=\begin{bmatrix} x^t &x_\infty\end{bmatrix}^t$, then
(by construction)
the first $d'$ entries of $\check v(\check x)$ coincide with $v(x)$,
that is,
$$\check v(\check x)_{i}=v(x)_i,\quad i=1,\ldots,d'. $$
Additionally,
\begin{equation}\label{eq:addit}
0=\dot{\check x}_{d+1}
= \check v(\check x)_{\ell}- \check v(\check x)_{d'+1}
\end{equation}
so
we obtain
\begin{equation}\label{eq:addit2}
\check v(\check x)_{d'+1}=\check v(\check x)_{\ell}=v(x)_\ell.
\end{equation}
Note that by construction, $\check v(\check x)_{d'+1}$ depends only on
$x_\infty$ and thus we can solve \eqref{eq:addit2} for $x_\infty$ uniquely.
Hence
\begin{align*}
0 & = \dot{\check x}_{q}= \sum_{i=1}^{d'+1}
 \check S_{qi} \check v(\check x)_i
 =
\sum_{\dd{i=1}{i\neq\ell}}^{d'} S_{qi} v(x)_i
+ S_{q\ell} \check v(\check x)_{d'+1} \\
&=
\sum_{\dd{i=1}{i\neq \ell} }^{d'} S_{qi} v(x)_i + S_{q\ell} v(x)_\ell=
\sum_{i=1}^{d'} S_{qi} v(x)_i.
\end{align*}
For $s\neq q$,
$$
0=  \dot{\check x}_{s}=\sum_{i=1}^{d'+1} \check S_{si} \check v(\check x)_i
 =
 \sum_{i=1}^{d'}  S_{si} v(x)_i
$$
proving $Sv(x)=0$.
(Alternatively, the conclusion can be reached by the proof of
Proposition \ref{prop:ker}.)
The calculation above reverses to show
that converses of these implications hold as well.
\end{proof}

\begin{theorem}\label{thm:eq}
Let $S$ be a stoichiometric matrix corresponding to
a chemical reaction network.
Then the reaction form differential equations corresponding to $S$
and to its sign fixing
matrix $\wh S$
under mass action kinetics
have the  equilibria which are equivalent under the correspondence in
Corollary {\rm\ref{cor:eq}}.
\end{theorem}

\begin{proof}
This follows easily from Corollary \ref{cor:eq}
and Remark \ref{rem:sflessbad} by an induction
on the number of bad submatrices of $S$.
\end{proof}

\begin{rem}
Theorem \ref{thm:eq} and Corollary \ref{cor:eq} extend
to more general, reaction form ODEs \eqref{eq:chemde}
with  monotone fluxes.
In one step of the algorithm the key is to add a reaction consuming
exactly one (new) species  (variable) $x_\infty$.
Since this is an artificial
reaction we can  specify a flux $\check v(\check x)_{d'+1}$
and the key is to pick it to be monotone and
\emph{surjective}, e.g.~it depends only on $x_\infty$ and is linear.
This  ensures the solvability of
\eqref{eq:addit2} for $x_\infty$.
The uniqueness of $x_\infty$ is then guaranteed by the monotone property.
Under these assumptions both proofs work verbatim.\qed
\end{rem}

\subsection{Local stability is preserved by sign fixing}

In the previous subsection we showed that the equilibria
of the original CRN sit in a perfect correspondence with
those of the sign fixed CRN.
An important question is whether or not stability of an
equilibrium of the original CRN implies stability of the
corresponding equilibrium of the sign fixed CRN.
This question is open to interpretation
because the sign fixing CRN contains a rate constant which we
are allowed to define. Let us call this rate constant $k$.
A natural version of the question would be: is there an \emph{a priori}
choice of $k$ such that the equilibrium of the original CRN is stable
if and only if the corresponding equilibrium is stable for the sign fixed
CRN.
While we have not analyzed global stability, we have analyzed and answered
the question for \emph{local} asymptotic stability.
We found that if we choose $k$ large enough, then
one of the eigenvalues of the sign fixed Jacobian will be very negative,
and all the others will be close to the eigenvalues of the Jacobian of
the original CRN.
Recall that a matrix is said to be {\bf stable} if all its eigenvalues
have negative real part.
An
equilibrium $x_0\in\RR^{d}_{\geq 0}$
of an ODE of the form \eqref{eq:de} is {\bf locally asymptotically stable} if the matrix $f'(x_0)$ is stable.

As before, we assume mass action kinetics although
this assumption can be weakened to reaction form ODEs \eqref{eq:chemde}
with  monotone entrywise surjective fluxes.

\begin{theorem}\label{thm:lstab1}
Let $S$ be a $d\times d'$ stoichiometric matrix with a bad submatrix.
Let $\check S$ be obtained from $S$ by applying the
sign fixing algorithm to eliminate this bad submatrix.
Write $J(x)=Sv'(x)$ and $\check J_k(\check x)=\check S \check v'(\check x)$.
Here $k$ denotes the rate constant assigned to the additional reaction
created in the sign fixing algorithm. Fix a point $\check x\in\RR_{\geq 0}^{d+1}$ and let $x\in\RR_{\geq 0}^d$
denote its first $d$ components. Furthermore, let $J=J(x)$ and
$\check J_k=\check J_k(\check x)$.

Then $d$ of the eigenvalues of the $(d+1)\times (d+1)$ Jacobian matrix
$\check J_k$
$($counting multiplicity$)$
converge $($as $k\to\infty)$
to the $d$ eigenvalues of $J$ and
the remaining eigenvalue is real and converges to $-\infty$.
\end{theorem}

Without loss of generality, we may assume the bad submatrix in
$S$ is the $2\times 2$ bottom right block and
$S_{d,d'}>0$.
Then the relationship between the $d\times d$ matrix
$J=Sv'(x)$ and the $(d+1)\times (d+1)$ matrix
$\check J_k=\check S\check v'(\check x)$
is as follows:
\begin{eqnarray*}
\check J_k & =& \begin{bmatrix}
J_{1,1}& \cdots & J_{1,d} & 0\\
\vdots & \ddots & \vdots & \vdots \\
J_{d-1,1}& \cdots & J_{d-1,d} & 0\\
J_{d,1}- S_{d,d'}\frac{\partial v_{d'}}{\partial x_1} & \cdots & J_{d,d}-
 S_{d,d'}\frac{\partial v_{d'}}{\partial x_d}& k S_{d,d'} \\
 \frac{\partial v_{d'}}{\partial x_1} & \cdots &
 \frac{\partial v_{d'}}{\partial x_d}& -k  
\end{bmatrix}
\\
&=&
\begin{bmatrix}
& && 0\\
& J && 0\\
&&&0 \\
0&0&0&0
\end{bmatrix}
+ \begin{bmatrix}
0\\
0\\
-S_{d,d'}\\
1
\end{bmatrix}
\begin{bmatrix}
\frac{\partial v_{d'}(x)}{\partial x_1} & \cdots
& \frac{\partial v_{d'}(x)}{\partial x_{d}} & -k
\end{bmatrix} 
.
\end{eqnarray*}

\def\CC{\mathbb C}

Let $c(\la)
=\det (J- \la I_d)\in\RR[\la]$ and $c_k(\la)
=\det (\check J_k-\la I_{d+1})\in\RR[\la]$
 denote the characteristic polynomials (in $\lambda$) of
$J$ and $\check J_k$, respectively.

\begin{lem}\label{lem:charPoly}
The degree $d+1$ polynomials $\frac{1}{k}c_k$ converge
uniformly on compact
subsets of $\CC$
to the degree $d$ polynomial $-c$.
\end{lem}

\begin{proof}
Let us consider $\check J_k - \la I$ and its determinant. For notational
convenience let us write $s=S_{d,d'}$ and $v_{d',j}=
\frac{\partial v_{d'}}{\partial x_j}$. Then
\begin{eqnarray*}
c_k=
\det (\check J_k - \la I) &=&
\det
\begin{bmatrix}
J_{1,1}-\la& \cdots & J_{1,d} & 0\\
\vdots & \ddots & \vdots & \vdots \\
J_{d-1,1}& \cdots & J_{d-1,d} & 0\\
J_{d,1}- s v_{d',1} & \cdots &J_{d,d}-s v_{d',d} -\la & k s\\
   v_{d',1} &  \cdots &
    v_{d',d}& -k -\la \\
   \end{bmatrix} \\
   &=&
   \det
\begin{bmatrix}
J_{1,1}-\la& \cdots & J_{1,d} & 0\\
\vdots & \ddots & \vdots & \vdots \\
J_{d-1,1}& \cdots & J_{d-1,d} & 0\\
J_{d,1} & \cdots &J_{d,d} -\la &  -s\la \\
   v_{d',1} &  \cdots &
      v_{d',d}& -k -\la \\
        \end{bmatrix}\\
        &=&
        (-k-\la) c + s \la \, 
\underbrace{
\det
\begin{bmatrix}
J_{1,1}-\la& \cdots& J_{1,d-1}  & J_{1,d} \\
\vdots & \ddots & \vdots & \vdots  \\
J_{d-1,1}& \cdots & J_{d-1,d-1}-\la & J_{d-1,d} \\
   v_{d',1} &  \cdots & v_{d',d-1} &
      v_{d',d} \\
        \end{bmatrix}}_h
\end{eqnarray*}
Thus
\beq
\label{eq:detJac}
\det (\check J_k - \la I)
= c_k
=
(-k-\la) c + s \la h
=
(-k-\la) \; \det J \ + \ s \la h.
\eeq
Note $h\in\RR[\la]$ is a polynomial of degree
$\leq d-1$ in $\la$ and does not contain $k$.
Thus
$$
\frac 1k c_k  = \  \frac {-k-\la}kc \ - \
\frac 1ks \la h \ \xrightarrow{k\to\infty} \ -c
$$
uniformly on compact subsets of $\CC$.
\end{proof}

In fact, the polynomial $h$ from the proof of Lemma \ref{lem:charPoly}
is of degree $\leq d-2$.
Since the ODEs are in
reaction form, $S_{d,d'}>0$ implies
$v_{d',d}=\frac{\partial v_{d'}}{\partial x_d} =0$,
cf.~\eqref{eq:vdepjacI}.

\begin{proof}[Proof of Theorem {\rm\ref{thm:lstab1}}]
Let $x_j^k, m_j^k$ denote the zeroes of $c_k$
together with their multiplicities and $x_j, m_j$ denote
the zeroes of $c$.
Certainly $c_k$ is analytic in the complex variable $\la$,
thus
$d$ zeroes of $c_k$ (counting multiplicity)
converge to the zeroes  of $c$.
This is a standard consequence of the argument principle,
since we can put a small circle $C_\eps $ around
a zero of $x_j$ and for large enough $k$
the winding number (with respect to $0$)
of  $\frac{1}{k} c_k $ on  $C_\eps $
equals that of $c$.
Thus  $c$ and $c_k $ have the same number of zeroes
inside  $C_\eps $.

Similarly, to analyze the point at infinity,
one can draw a circle $C_R$ of arbitrarily large radius $R$
containing all zeroes of $c$.
The winding number (with respect to $0$) of $c$ around $R$
is $d$, so for large enough $k$ the winding number of $\frac{1}{k} c_k$
is also $d$, thus one zero of $c_k$,
without loss of generality denote it $ x_{d+1}^k$ lies outside of $C_R$.
Hence the sequence $ x_{d+1}^k$ diverges to infinity.
Since all coefficients of the polynomial $c_k$ are real,
its zeroes are either real or occur in conjugate pairs.
So $ x_{d+1}^k$ must be real,
since if not $c_k$ would have two zeroes outside of $C_R$.

Let us retain the notation from the proof of Lemma \ref{lem:charPoly}.
Then
$$c_k=(-k-\la) c+s \la h\in\RR[\la],$$
where $c\in\RR[\la]$ is of degree $d$,
$h\in\RR[\la]$ is of degree $\leq d-1$, and $s\in\RR$.
Thus for $\la>0$ big enough, $-\la c$ dominates $s \la h$.
For such $\la>0$ the sign of $c_k(\la)$ will equal
the sign of $-c(\la)$ for any $k>0$. This shows that with $R$ big,
the zero $x_{d+1}^k$ of $c_k$ outside of $C_R$ must be negative,
thereby concluding the proof.
\end{proof}

\begin{theorem}\label{thm:lstab2}
Let $S$ be a stoichiometric matrix corresponding to a CRN
and let $\wh S$ be one of its
sign fixing matrices. Then there exists a choice of rate constants
for the added reactions such that the equilibria of
the reaction form ODEs corresponding to $S$ are locally
 asymptotically stable if and only if the same holds
for the equilibria of
the reaction form ODEs corresponding to $\wh S$.
\end{theorem}

\begin{proof}
This follows easily from Theorem \ref{thm:lstab1} by an induction on the
number of bad submatrices in $S$.
\end{proof}

\subsection{Craciun-Feinberg theory: determinants of Jacobians}

This brief subsection is for those familiar with
the Craciun-Feinberg (CF) theory \cite{CF05,CF06}
and we observe that it behaves well under sign fixing.
Recall the key requirement of the CF theory in order to invoke
its consequences is
that the determinant $\Delta$ of the Jacobian has an unambiguous
sign.
Note from \eqref{eq:detJac} with $\lambda=0$ that
the determinant of the Jacobian of the sign fixed CRN
is a scalar times $\Delta$.
Thus one determinant has an unambiguous sign if and only if the other
one does.

The conclusion is that if the CF theory applies to a CRN,
then it applies to the closely related
CRN whose Jacobian respects a sign pattern.

\section{Deficiency vs.~sign patterns}
\label{sec:defic}

An important notion in chemical networks is that
of {\bf deficiency}. In this section we show that sign fixing
might increase the deficiency of a CRN by at the most the number
of bad submatrices for the stoichiometric matrix of the original CRN.

We follow the notation and terminology
of Gunawardena \cite{Gu} (or see \cite{Fein79,CDSS}).
 Thus, we denote:
 \begin{eqnarray*}
  n&:=& \text{the number of complexes of the network},\\
  \ell &:=& \text{the number of linkage classes of the network},\\
     s &:=& \text{the rank of the stoichiometric matrix},
     \end{eqnarray*}
and the {\bf topological deficiency} of the network is
$$
\delta:=n-\ell-s.
$$

\subsection{Zero deficiency vs.~sign patterns}
\label{sec:zerodef}

A natural question is whether the sign pattern
 of $f'(x)=Sv'(x)$ has
 any correlation to the CRN having zero deficiency.
 The
 answer is no, and in this subsection we give examples of
 \begin{enumerate}
\item
chemical networks $S$
 with zero deficiency and no sign pattern for $Sv'(x)$;
 \item
chemical networks $S$ with nonzero deficiency and a sign pattern for
$Sv'(x)$.
\end{enumerate}

\begin{exam}[see \protect{\cite[\S 4.3]{Ka}} for more details]
Consider the reaction network
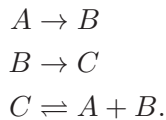
\begin{align*}
A &\rightarrow B \\
B &\rightarrow C \\
C &\rightleftharpoons A+B.
\end{align*}
The deficiency of the CRN is easily seen to be zero.
However,
$$ S =  \begin{bmatrix}
         -1 & 0 & -1 & 1\\
                    1 & -1 & -1 & 1 \\
                                     0 & 1  & 1  & -1
                                                       \end{bmatrix}  $$
so $Sv'(x)$ will not respect a sign pattern (Theorem \ref{thm:super}). It will have exactly
one entry without a sign. In order to obtain an example of a
deficiency zero network with an arbitrary number of non-signed entries
in the Jacobian, one simply considers a network with the
following stoichiometry:
$$
\begin{bmatrix}
S \\
& S \\
&&\ddots \\
&&& S
\end{bmatrix}.
$$
                   \end{exam}

Conversely, having a sign pattern will not yield any information about
the deficiency of the network.

\begin{exam}
Consider the chemical reaction network
$$ B+C \rightleftharpoons  A  \rightleftharpoons  B'+C'
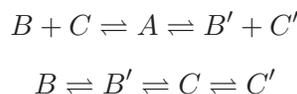
$$
$$B \rightleftharpoons  B'\rightleftharpoons  C \rightleftharpoons  C'$$
with stoichiometric matrix
$$
S=\begin{bmatrix}
 -1 & -1 & 0 & 0 & 0 & 1 & 1 &
   0 & 0 & 0 \\
 1 & 0 & -1 & 0 & 0 & -1 & 0 &
   1 & 0 & 0 \\
 0 & 1 & 1 & -1 & 0 & 0 & -1 &
   -1 & 1 & 0 \\
 1 & 0 & 0 & 1 & -1 & -1 & 0 &
   0 & -1 & 1 \\
 0 & 1 & 0 & 0 & 1 & 0 & -1 &
   0 & 0 & -1
\end{bmatrix}.
$$
By Theorem \ref{thm:super}, $Sv'(x)$ respects an unambiguous sign pattern but
the deficiency of the network is one.
To achieve arbitrary
deficiency one can employ a block diagonal construction as above.
\end{exam}

\subsection{Deficiency and the sign fixing algorithm}
\label{sec:deficSFix}

In this subsection we
consider how the deficiency of a CRN changes
after we apply the sign fixing algorithm to produce a new CRN.

Let $S_1$ be the stoichiometric matrix for
a CRN, and let $S_2$ be
the sign fixing matrix of $S_1$ with respect to some bad submatrix.
All variables with subscript 1 refer to
the original CRN and variables with subscript 2 refer to
the new CRN unless otherwise noted.
Also, $\mathcal{C}$ denotes the set of all complexes of a network
and $\mathcal{L}$ denotes the set of all linkage classes of a network.
We also denote
$\Delta \delta := \delta_2-\delta_1$,
$\Delta n := n_2 - n_1$, $\Delta \ell := \ell_2 - \ell_1$,
and $\Delta s = s_2 - s_1$.
Assuming that $S_1$ has a bad submatrix corresponding to the species $A,B$,
then this network has 2 reactions of the form
\begin{align}
                 \label{eq:badV1} p_1A+C_1 &\to p_2B + C_2 & \\
                 \label{eq:badV2}                 p_3A+p_4B+C_3 & \to C_4,
                                     \end{align}
where $p_1,p_2,p_3,p_4 \in \mathbb{N}$ and $C_1,C_2,C_3,C_4$ are some
(possibly empty) positive linear combination of species.
The only changes to the new network are we add a new species $B'$, reaction \eqref{eq:badV1} is replaced by
\begin{equation}\label{eq:badV1'}
p_1A+C_1 \rightarrow B' + C_2
\end{equation}
and we create an additional reaction
\begin{equation}\label{eq:badV2'}
B' \rightarrow p_2B.
\end{equation}
By Proposition \ref{prop:ker}, we always have $\Delta s = 1$.
To get a better handle on the change of deficiency,
we proceed as follows.

\begin{lemma}\label{lem:v1}
If $S_1$ is the stoichiometric matrix of a CRN and $S_2$ is its sign fixing matrix with respect to a bad submatrix, then the following inequalities are sharp:
\begin{equation}\label{eq:v1}
\Delta \ell \leq 2 \text{  and  } 1 \leq \Delta n \leq 3.
\end{equation}
\end{lemma}

\begin{proof}
With the notation above, the only possible new complexes are $B' + C_2, B'$, and $p_2B$. Hence $\Delta n \leq 3$.
Also, the only possibly new linkage classes are $[B' + C_2]_2$ and $[p_2B]_2$,
so $\Delta \ell \leq 2$.
The fact that $\Delta n \geq 1$ is obvious.

To show that the inequalities \eqref{eq:v1} are sharp,
consider the following network:
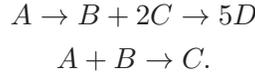
\begin{eqnarray*}
& A \rightarrow B+2C \rightarrow 5D & \\
& A+B \rightarrow C.&
\end{eqnarray*}
This CRN has the stoichiometric matrix
$$S_1 = \begin{bmatrix}
    -1 & -1 & 0 \\
    1 & -1 & -1 \\
    2 & 1 & -2 \\
    0 & 0 & 5
    \end{bmatrix}.$$
    Hence $n_1=5$ and $\ell_1 = 2$. The sign fixing matrix for $S_1$ (with respect
to the species $A,B$) is
    $$S_2 = \begin{bmatrix}
    -1 & -1 & 0 & 0\\
    0 & -1 & -1 & 1\\
    2 & 1 & -2& 0\\
    0 & 0 & 5 & 0\\
    1 & 0 & 0 & -1
    \end{bmatrix},$$
    and our new chemical network is
    \begin{eqnarray*}
    & A \rightarrow 2C+B', \ A+B \rightarrow C & \\
& B+2C \rightarrow 5D,\, B' \rightarrow B. &
\end{eqnarray*}
We thus see that $n_2=8$ and $\ell_2=4$, so $\Delta n = 3$ and $\Delta \ell = 2$.
\end{proof}

We now need an efficient way to determine which complexes
and reactions of the new CRN affect $\Delta n$ and $\Delta \ell$.
We shall define functions which will precisely determine which complexes
and linkage classes in the new CRN increase $\Delta n$ and $\Delta \ell$.
Define $\phi:\{B'+C_2,B',p_2 B\} \rightarrow \{0,1\}$ by
\begin{eqnarray*}
\phi(B'+C_2)&=&
\begin{cases}
                    2 & \text{if } C_2 \neq \emptyset \text{ and } p_2 B +C_2 \in \mathcal{C}_2  \\
                   1 & \text{otherwise}
                    \end{cases} \\
\phi(p_2 B) &=& \begin{cases}
                            1 & \text{if } p_2 B \not \in \mathcal{C}_1 \\
                            0 & \text{otherwise}
                            \end{cases}.
\end{eqnarray*}
We also define $\psi: \{[B'+C_2]_2, [B']_2\} \rightarrow \{0,1\}$ such that
\begin{eqnarray*}
\psi([B'+C_2]_2) &=& \begin{cases}
                    1 & \text{if } p_2 B +C_2 \in \mathcal{C}_2 \text{ and } [B'+C_2]_2 \bigcap [p_2 B+C_2]_2 = \emptyset \\
                   0 & \text{otherwise}
                    \end{cases}\\
                    \psi([B']_2)&=&
\begin{cases}
            1 & \text{if } p_2 B \not \in \mathcal{C}_1 \text{ and } C_2 \neq \emptyset \\
            0 & \text{otherwise}
            \end{cases}.
            \end{eqnarray*}
The advantage of this new notation is that we now have
a succinct way to measure $\Delta n$ and $\Delta \ell:$
\begin{eqnarray}
\label{eq:v3}
\Delta n &=& \phi(B'+C_2) + \phi(p_2 B) \\
\label{eq:v4}
\Delta \ell &=& \psi([B'+c_2]_2) + \psi([B']_2).
\end{eqnarray}
Equation \eqref{eq:v4}
follows directly from the definition of $\Delta \ell$
and the construction of the new network.
However, \eqref{eq:v3} needs more justification.

\begin{lemma}\label{lem:v2}
With the setup described above, \eqref{eq:v3} holds.
\end{lemma}

\begin{proof}
By construction, notice that we always have
$\mathcal{C}_1 \subseteq \mathcal{C}_2$ or
$\mathcal{C}_1 \setminus \{p_2B+C_2\} \subseteq \mathcal{C}_2$.
Also, $\{B'+C_2, B', p_2B\}$ are the only possible
new complexes that are not in $\mathcal{C}_1$.
Observe that if $p_2B+C_2 \in \mathcal{C}_2$,
then $\mathcal{C}_2=\mathcal{C}_1 \bigcup \{B'+C_2,B',p_2B\}$, whence
\begin{eqnarray*}
\Delta n&=&
\text{card}(\mathcal{C}_1 \cup \{B'+C_2, B', p_2B\}) - \text{card}(\mathcal{C}_1) \\
&=&\big(\text{card}(\mathcal{C}_1 \cup \{B'+C_2, B'\}) - \text{card}(\mathcal{C}_1)\big) + \big(\text{card}(\mathcal{C}_1 \cup \{p_2B\}) - \text{card}(\mathcal{C}_1)\big) \\
&=& \text{card}(\mathcal{C}_1 \cup \{B'+C_2, B'\}) - \text{card}(\mathcal{C}_1)+ \phi(p_2B),
\end{eqnarray*}
where the last equality follows by the definition of $\phi$.
Also, if $C_2 = \emptyset$, then
$$\text{card}(\mathcal{C}_1 \cup \{B'+C_2, B'\})
- \text{card}(\mathcal{C}_1)=1; \text{otherwise},
\text{card}(\mathcal{C}_1 \cup \{B'+C_2, B'\})
- \text{card}(\mathcal{C}_1)=2,$$
so
$\phi(B'+C_2)= \text{card}
(\mathcal{C}_1 \cup \{B'+C_2, B'\}) - \text{card}(\mathcal{C}_1)$
by construction.
This implies $\Delta n = \phi(B'+C_2)+\phi(p_2B)$, as desired. On the other hand, suppose $p_2B+C_2 \not \in \mathcal{C}_2$. If $C_2=\emptyset$, then $p_2B+C_2=p_2B \in \mathcal{C}_2$, which is a contradiction. Thus we must have $C_2 \neq \emptyset$. A simple count shows
\begin{eqnarray*}
\Delta n &=& \text{card}(\mathcal{C}_1 \setminus \{p_2B+C_2\} \cup \{B'+C_2, B', p_2B\}) - \text{card}(\mathcal{C}_1)= \begin{cases}
            2 & \text{if } p_2 B \not \in \mathcal{C}_1 \\
            1 & \text{otherwise}
            \end{cases}\\
            &=& \phi(B'+C_2) + \phi(p_2B),
\end{eqnarray*}
where the last equality follows directly from the definition of $\phi$.
\end{proof}

\begin{theorem}\label{thm:v1}
Let $S_1$ be the stoichiometric matrix to a chemical network with a bad submatrix, and let $S_2$ be the sign fixing matrix with respect to this bad submatrix. Then $0 \leq \Delta \delta \leq 1,$ and this inequality is sharp.
\end{theorem}

\begin{proof}
First, Lemma \ref{lem:v1}
shows $\Delta \delta = \Delta n- \Delta \ell - 1 \leq 3-\Delta \ell - 1 \leq 2$. Notice that if $\Delta \delta=2,$ then $\Delta n =3$ and $\Delta \ell = 0$, so
$\phi(B'+C_2)=2$, $\phi(p_2 B)=1$ and $\psi([B'+C_2]_2) = \psi([B']_2)=0$.
However, $\phi(B'+C_2)=2$ implies $C_2 \neq \emptyset$,
and $\phi(p_2B)=1$ implies $p_2B \not \in \mathcal{C}_1$ by construction.
Hence $\psi([B']_2)=1$ by the definition of $\psi$ and this is a contradiction. Thus, we cannot
 have $\Delta n=3$ and $\Delta \ell=0.$ This proves $\Delta \delta \leq 1$.

 Now suppose $\Delta \delta < 0$ to derive a contradiction. By Lemma \ref{lem:v1}, $\Delta n=2$ and $\Delta \ell = 2,$ or $\Delta n = 1$ and $\Delta \ell \geq 1$.

{\sc Case 1}: Suppose $\Delta n = 2$ and $\Delta  \ell =2 $. Notice if
$\Delta \ell = 2$ then $\psi([B'+C_2]_2)=\psi([B']_2)=1$ and thus
$C_2 \neq \emptyset,$ $p_2B+C_2 \in \mathcal{C}_2,$ and
$p_2 B \not \in \mathcal{C}_1$ by construction. Hence
$\phi(p_2 B)=1$ and $\phi(p_2B+C_2)=2$ by construction, so $\Delta n=2+1=3$ by
\eqref{eq:v3} and this is a contradiction.

{\sc Case 2}: Suppose $\Delta n =1$ and $\Delta \ell \geq 1$.
Notice, $\Delta n =1$ and \eqref{eq:v3} imply $\phi(p_2 B)=0$ and
$p_2B \in \mathcal{C}_1,$ so $\psi([B']_2)=0$ and $\psi([B'+c_2]_2)=1$ since $\Delta \ell \geq 1$ implies $p_2B+C_2 \in \mathcal{C}_2$ and $[B'+C_2]_2 \cap [p_2 B+C_2]_2 = \emptyset$.
  If $C_2 = \emptyset$, then $[B'+C_2]_2 = [B']_2 = [p_2B]_2=[p_2B+C_2]_2$. Hence $ [B'+C_2]_2 \cap [p_2 B+C_2]_2 \neq \emptyset$ and this is a contradiction. Thus $C_2 \neq \emptyset$. But then $\phi(B'+C_2)=2$ by construction, and $\Delta n \geq 2$ by \eqref{eq:v3}; contradiction.

To show that these inequalities are sharp, consider the following network:
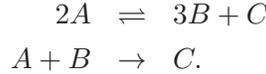
\begin{eqnarray*}
 2A &\rightleftharpoons& 3B+C\\
A+B &\rightarrow& C.
\end{eqnarray*}
This CRN has the stoichiometric matrix
$$S_1 = \begin{bmatrix}
    -2 & -1 & 2 \\
    3 & -1 & -3 \\
    1 & 1 & -1
    \end{bmatrix}.$$
    Hence $n_1=4$ and $\ell_1 = 2$. The sign fixing matrix for $S_1$ with
    respect to $A,B$ is
    $$S_2 = \begin{bmatrix}
    -2 & -1 & 2 & 0\\
    0 & -1 & -3 & 3\\
    1& 1 & -1& 0\\
    1 & 0 & 0 & -1
    \end{bmatrix},$$
    and our new chemical network is
    \begin{eqnarray*}
 3B+C &\rightarrow& 2A \rightarrow B'+C\\
 A+B& \rightarrow& C\\
B'& \rightarrow& 3B.
\end{eqnarray*}
We thus see that $n_2=7$ and $\ell_2=3$. So $ \Delta n = 3$ and $\Delta \ell = 1$.
Thus, $\Delta \delta = \Delta n - \Delta \ell - \Delta s = 3-1-1=1.$
\end{proof}

\begin{cor}\label{cor:badS}
Suppose $S$ is the stoichiometric matrix for some chemical network, and it contains $k$ bad submatrices. If $\widehat{S}$ is its sign fixing matrix, then $0 \leq \Delta \delta \leq k$.
\end{cor}

\begin{proof}
Recall that by definition, $\widehat{S}$ is determined by a recursive sequence of at most $k$ sign fixing matrices, each with respect to a certain bad submatrix from the previous matrix in the sequence.
 An application of Theorem \ref{thm:v1} at each step yields our desired result.
\end{proof}

In fact, the upper bound for $\Delta\delta$ in Corollary \ref{cor:badS}
is the number of equivalence
classes of bad submatrices of $S$ as defined in \S \ref{subsec:nonuniq}.
Also, all sign fixed matrices $\wh S$ obtained from our algorithm have the
same deficiency (independent
of the order in which the sign fixing algorithm is applied).

Theorem \ref{thm:v1} together with formulas \eqref{eq:v3} and \eqref{eq:v4}
helps determining necessary conditions for $\Delta \delta =1$ for a stoichiometric matrix and its sign fixing matrix with respect to a certain bad submatrix.
Notice that $\Delta \delta =1$ implies $\Delta n =2$ and $\Delta \ell=0$ or $\Delta n=3$ and $\Delta \ell =1$. If $C_2 = \emptyset$, then $p_2B=p_2B+C_2 \in \mathcal{C}_1$, so $\phi(p_2B)=0$. Also, $C_2 = \emptyset$ implies  $\phi(B'+C_2)=1$ by construction, So $\Delta n=1$ by \eqref{eq:v3}, and hence,
$\Delta \delta =0$ by Theorem \ref{thm:v1}. This observation yields the following:

\begin{theorem}
Let $S$ be the stoichiometric matrix for a chemical network. Suppose that the column corresponding to each bad submatrix of $S$ with the positive entry has only one positive entry. Then if $\widehat{S}$ is a sign fixing matrix for $S$, we have $\Delta \delta =0$.
\end{theorem}

\begin{proof}
By assumption, each bad submatrix of $S$ corresponds to 2 reactions of the form:
\begin{align}
                 \label{eq:badV1''} p_1A+C_1 &\to p_2B  & \\
                 \label{eq:badV2''}               p_3A+p_4B+C_2 & \to C_3,
                                     \end{align}
 where $A,B$ are species, $p_1,p_2,p_3,p_4 \in \mathbb{N}$ and $c_1,c_2,c_3$ are some (possibly empty) positive linear combination of species.
 As already shown, the deficiency for the sign fixing matrix of $S$ with respect to this bad submatrix does not change. An inductive procedure yields our desired result.
\end{proof}

\section{An alternative sign fixing algorithm?}
\label{sec:signfix}

Given a stoichiometric matrix $S$ with bad submatrices, there is an
easier way of eliminating these. Instead of performing the sign
fixing algorithm
for each submatrix separately and thus adding a row and a column
in every step, we can add only one row and column and eliminate all
bad submatrices in a single step.
Unfortunately, this construction changes the dimension of $\ker S$;
thus the two matrices yield reaction networks
with very different equilibria structure.
We illustrate this with an example.

\begin{exam}
Suppose
$$S = \begin{bmatrix}
    -2&-1&4&4&-4\\
    -12&4&4&0&0\\
    4&-1&-2&0&0\\
    10&-2&-6&-4&4
    \end{bmatrix}.$$
Notice that $S$ has several bad submatrices.
We start by adding a row and column of zeros to $S$.
Pick a bad $2\times 2$ submatrix
of $S$. Replace the positive entry $S_{p\ell}$ of $S$ by $0$, add $+1$ to the
$\ell$th entry of the new row and add $S_{p\ell}$ to the
$p$th entry of the new column. Repeat this for all the
bad submatrices.
After all the bad submatrices have been eliminated, the bottom right
entry is changed into the negative sum of all the entries in the last row.
We obtain a matrix $\wt S$ with no bad submatrices. In
our example this is
$$
\wt S=\begin{bmatrix}
 -2 & -1 & 0 & 0 & -4 & 8 \\
  -12 & 0 & 4 & 0 & 0 & 4 \\
   0 & -1 & -2 & 0 & 0 & 4 \\
    0 & -2 & -6 & -4 & 0 & 14 \\
     1 & 1 & 1 & 1 & 1 & -5
\end{bmatrix}.
$$

We note that
$$\Ker S= \Span \left\{ \begin{bmatrix}
0&0&0&1&1 \end{bmatrix}^t , \begin{bmatrix}
1&2&1&0&0
\end{bmatrix}^t
\right\}
$$
and
$$
\Ker S^t= \Span \left\{ \begin{bmatrix}
1&1&1&1
\end{bmatrix}^t\right\}.
$$
The corresponding kernels for $\wt S$ are:
$$
\ker \wt S= \Span \left\{ \begin{bmatrix}
2&0&4&1&3&2
\end{bmatrix}^t\right\}
$$
and
$$
\ker \wt S^t=\{0\}.
$$

Given its kernel, the ODE $\dot {\wt x}= \wt S \wt v(\wt x)$ can have
equilibria only on the boundary. In fact, each solution
to $ \wt S \wt v(\wt x)=0$
can be shown to satisfy $\wt v(\wt x)=0$.

Also, assuming mass action
kinetics,
$$
v(x)=\begin{bmatrix}
k_1
x_1^2 x_2^{12}  & k_2 x_1 x_3
   x_4^2  & k_3 x_3^2 x_4^6 & k_4 x_4^4
          &k_5 x_1^4
      \end{bmatrix}^t
$$
so for every $x_2\in\RR_{>0}$,
$$
x_1=\frac{k_2
   \sqrt[4]{k_4}}{2 \sqrt{k_1}
   \sqrt{k_3} \sqrt[4]{k_5}
   x_2^6},\quad
x_3=\frac{4 k_1^{3/2} \sqrt{k_3}
   \sqrt[4]{k_4} x_2^{18}}{k_2^2
   \sqrt[4]{k_5}},\quad x_4=\frac{k_2}{2 \sqrt{k_1}
   \sqrt{k_3} x_2^6}
$$
yields a positive solution to $Sv(x)=0$.

Hence it is not possible to recover positive equilibria for the chemical
CRN by $S$ from those obtained by $\wt S$.

Also note that there is a nonnegative vector orthogonal
to the range of $S$, thus the corresponding reaction form dynamics
has a conserved quantity.
On the other hand, $\wt S$ is not conserving.
\end{exam}

\section{Software}
\label{sec:soft}

The discovering of the
results in this paper was considerably facilitated
by computer experiments. The programs we wrote to do this
might be of value to a broad community, so we documented them
and provided tutorial examples.
They are found on the web site\\
\centerline{\url{
http://www.math.ucsd.edu/~chemcomp/}}

The Mathematica files provided  contain software for
dealing with equations that come from CRNs;
$dx/dt=f(x)=Sv(x)$ as in \eqref{eq:chemde}.
Some of our commands  focus on the Jacobian, $f '$, of $f$;
they do the following
\ben[\rm (1)]
\item compute the Jacobian $f '$ of $f$ (given say the stoichiometric matrix $S$);
\item
check existence of a sign pattern for $f '(x)$ which remains unchanged
for all $x  \geq 0$, using Theorem \ref{thm:super} in this paper;
\item
implement the sign fixing algorithm in \S \ref{sec:algor};
\item compute the Craciun-Feinberg (CF) determinant \cite{CF05,CF06,CTF} of $f '$
(governs CRNs with outflows for all species with
outflow rate constants equal to one);
\item compute the more general Helton-Klep-Gomez core determinant \cite{HKG}
of $f '$
(governs CRNs with any number of outflows).
\een

The CF determinant and core determinants are used in tests to count
the number of positive equilibria, namely $ x^*>0$
such that $f(x^*)=0$.
(For more information, please look at the papers
 \cite{CHW} and \cite{HKG}
 and the original Craciun, Feinberg et al.~papers, \cite{CF05,CF06,CTF}.)

Another part of our Mathematica package deals with
deficiency of reaction form differential equations, as
discussed  in \S \ref{sec:defic}.
The software allows us to compute
the deficiency of a CRN as well
as conversion of representations as follows.
One starts with the traditional representation
$f(x) = Sv(x)$.
Our program produces the representation
$$ Sv(x) = YA_k \psi(x)$$
where $A_k$ is the Laplacian of the ``complexes graph''
of the chemical reaction network.
$Y$ is the matrix whose columns are indexed by complexes
and which contain nonzero entries corresponding
to chemical species which enter the complex.
$\psi$ is a list of monomials in the chemical concentrations.
For details, see \cite{Gu,Fein79,HJ}.

 Our commands also  compute  the components of the complexes graph.
Capability to automatically plot planar graphs
is under development and should be available soon.

\section{Conclusions}

\begin{enumerate}[\rm (1)]
\item
The effect of
any  species  on another species is always (for all concentrations)
 inhibitory  or always  excitatory,
if
(generically only if for reversible CRNs) the species-reaction
graph of $S$ has no
bad  cycle.
\item
If a CRN has bad cycles, then
{\it the sign fixing algorithm produces
a CRN whose Jacobian respects a  sign pattern}.
The key properties are:
\begin{enumerate}[\rm (a)]
\item
 Equilibria in perfect correspondence;
\item
 Local asymptotic stability corresponds perfectly;
\item
 Conservation laws correspond;
\item
The determinant of the sign fixed Jacobian is a constant
      multiple of the original Jacobian, thus the theory
      of Craciun and Feinberg applies to both CRNs or to neither;
\item
Deficiency does not drop and its increase has a simple bound.
\end{enumerate}
\end{enumerate}

Possibly the tests and constructions outlined here
will be useful to some experimentalists.
In an experiment where one aims to understand which reactions occur
and the pattern of inhibition and excitation,
one ``first"  obtains  the species-reaction graph $\cG$.
Based on measuring some limited number of concentrations
one determines the inhibitory or excitatory effects.
An issue is whether or not enough concentrations were measured,
for a possibility is that species $x_i$ has
an excitatory effect at some concentrations and an inhibitory
effect at others.

The results here give a simple way to
sort out  this problem.
First suppose the CRN does not contain two reactions involving two
species A and B, one species  being consumed by both reactions
while  the other  species is consumed by one reaction and
produced by the other. Then there is no such problem
(see Theorem \ref{thm:chemSP}); in principle, one concentration
measurement (of all species) suffices. If $\cG$ has such a bad pair
of reactions, probably there is trouble.
Our sign fixing algorithm shows additional measurements
which, if they  can be made,  fix this trouble.

We emphasize that our analysis does not say what effect a species $i$
has on species $j$ in a single reaction
but it bears on its effect inside the entirety of the CRN.

\section*{Acknowledgments}
\small
The authors thank Jan Schellenberger for his
expert assistance and Rohun Kshirsagar for discussions
and for computations.
We thank Gheorghe Craciun for his guidance through
literature.
We thank Eduardo Sontag for perspective and insight he has provided.

The software mentioned in \S \ref{sec:soft}  was produced by
 Igor Klep and  Bill Helton
and students  Karl Fredrickson and  Vitaly Katsnelson.
Mauricio de Oliveira,
Rohun Kshirsagar and Ruth Williams contributed suggestions.
The graph plotting capabilities are being implemented with the
help of Marko Boben.


\begin{thebibliography}{CTF06}

\bibitem[AnS]{AS}
D. Angeli, E.D. Sontag:
Monotone control systems,
IEEE Trans. Automat. Control {\bf 48} (2003) 1684--1698

\bibitem[ArS06]{ArS06}
M. Arcak, E.D. Sontag: 
Diagonal stability of a class of cyclic systems and its connection with
the secant criterion, Automatica {\bf 42} (2006) 1531–-1537

\bibitem[ArS08]{ArS08}
M. Arcak, E.D. Sontag: 
A passivity-based stability criterion for a class of interconnected
systems and applications to biochemical reaction networks,
Math. Biosci. Eng. {\bf 5} (2008) 1-–19

\bibitem[BQ]{BQ}
D.A.Beard, H. Qian: {\it Chemical Biophysics: Quantitative Analysis of Cellular Systems},
Cambridge Univ. Press, 2008

\bibitem[BS]{BS}
R.A. Brualdi, B.L. Shader:
{\it Matrices of sign-solvable linear systems},
Cambridge Univ. Press, 1995

\bibitem[CD]{CD}
O. Cinquin, J. Demongeot: Positive and negative feedback: Striking a
balance between necessary antagonists, J. Theor. Biol. 
\textbf{216} (2002) 229--241


\bibitem[CDSS]{CDSS}
G. Craciun, A. Dickenstein, A. Shiu, B. Sturmfels:
Toric dynamical systems,
preprint (2007)\\
\url{http://arxiv.org/abs/0708.3431}


\bibitem[CF05]{CF05}
G. Craciun, M. Feinberg:
Multiple equilibria in complex chemical reaction networks. I. The injectivity property,
SIAM J. Appl. Math. {\bf 65} (2005) 1526--1546

\bibitem[CF06]{CF06}
G. Craciun, M. Feinberg:
Multiple equilibria in complex chemical reaction networks. II. The species-reaction graph,
SIAM J. Appl. Math. {\bf 66} (2006) 1321--1338


\bibitem[CHW]{CHW}
G. Craciun, J.W. Helton, R.J. Williams:
Homotopy methods for counting reaction network equilibria,
Math. Biosci. {\bf 216} (2008) 140--149

\bibitem[CTF]{CTF}
G. Craciun, Y. Tang, M Feinberg:
Understanding bistability in complex enzyme-driven reaction networks,
Proc.~National Academy of Sciences {\bf 103}:23 (2006) 8697--8702


\bibitem[Fe]{Fein79}
M. Feinberg: Lectures on chemical reaction networks. Notes of
lectures given at the Mathematics Research Center of the University
of Wisconsin in 1979\\
\url{http://www.che.eng.ohio-state.edu/~FEINBERG/LecturesOnReactionNetworks}


\bibitem[Go]{Go} 
J.-L. Gouz\'e: Positive and negative circuits in dynamical systems, J.
Biol. Syst. \textbf{6} (1998) 11--15


\bibitem[Gu]{Gu}
J. Gunawardena:
Chemical Reaction Network Theory for in-silico Biologists,
preprint (2003)\\
\url{http://www.jeremy-gunawardena.com/papers/crnt.pdf}


\bibitem[HKG]{HKG}
J.W. Helton, I. Klep, R. Gomez:
Determinant Expansions of Signed Matrices and of Certain Jacobians,
accepted for publication in SIAM J. Matrix Anal. Appl.\\
\url{http://arxiv.org/abs/0802.4319}

\bibitem[HJ]{HJ}
F. Horn, R. Jackson:
General mass action kinetics,
Arch. Rational Mech. Anal. \textbf{47} (1972) 81--116

\bibitem[Ka]{Ka}
V. Katsnelson: {\it Chemical reaction networks: comparing
deficiency, determinant expansions, and sign patterns},
Undergraduate thesis, University of California at San Diego, 2008

\bibitem[Pa]{Pa}
B.\O. Palsson:
{\it Systems Biology: Properties of Reconstructed Networks},
Cambridge Univ. Press, 2006

\bibitem[So]{So}
C. Soul\'e: Graphic Requirements for Multistationarity, ComPlexUs
 \textbf{1} (2003) 123--133

 \bibitem[Th]{Th}
 R. Thomas: On the relation between the logical structure ofsystems and
 their ability to generate multiple steady states or sustained oscillations,
 Springer Ser. Synergetics \textbf{9} (1981) 180--193

\bibitem[TK]{TK} 
R. Thomas, M. Kaufman: Multistationarity, the basis of cell
differentiation and memory. I. Structural conditions of multistationarity
and other nontrivial behaviour, Chaos \textbf{11} (2001) 170--179


\end{thebibliography}
\end{document}